\documentclass[12pt,a4paper]{amsart}

\usepackage{times}
\usepackage{amsmath}
\usepackage{amscd}
\usepackage{amssymb}
\usepackage[latin2]{inputenc}
\usepackage{t1enc}
\usepackage[mathscr]{eucal}
\usepackage{indentfirst}
\usepackage{graphicx}
\usepackage{graphics}
\usepackage{pict2e}
\usepackage{epic}
\numberwithin{equation}{section}
\usepackage[margin=2.9cm]{geometry}
\usepackage{epstopdf} 
\usepackage{mathtools}
\usepackage{xcolor}
\usepackage{hyperref}
\usepackage{bm}
\usepackage{bbm}
\usepackage{amsopn}

\long\def\remove#1{}

\newtheorem{theorem}{Theorem}[section] 
\newtheorem{lemma}[theorem]{Lemma}
\newtheorem{claim}[theorem]{Claim}

\newtheorem{definition}{Definition}[section]

\newcommand {\mm}[1] {\ifmmode{#1}\else{\mbox{\(#1\)}}\fi}
\newcommand{\denselist}{\itemsep 0pt\parsep=1pt\partopsep 0pt}

\newcommand{\myER}			{{Erd\H{o}s--R\'enyi}}

\newcommand{\myE}			{{\mathbb{E}}}
\newcommand{\myVar}[1]			{\mathrm{Var} \left[ #1 \right]}
\newcommand{\myi}			{{\bm{i}}}
\newcommand{\myj}			{{\bm{j}}}
\newcommand{\myk}			{{\bm{k}}}
\newcommand{\myprob}[1]				{\mathbb{P} \left[ #1 \right]}

\DeclarePairedDelimiter\floor{\lfloor}{\rfloor}

\begin{document}
\title[The $1$st Betti number of layer-$1$ subgraphs]{A limit theorem for the $1$st Betti number of layer-$1$ subgraphs in random graphs}
\author{Minghao Tian}
\address{Computer Science and Engineering Dept., The Ohio State University, USA}
\email{tian.394@osu.edu}

\author{Yusu Wang}
\address{Computer Science and Engineering Dept., The Ohio State University, USA}
\email{yusu@cse.ohio-state.edu}

\date{\today}

\begin{abstract}
We initiate the study of local topology of random graphs. The high level goal is to characterize local ``motifs'' in graphs. In this paper, we consider what we call the \emph{layer-$r$ subgraphs} for an input graph $G = (V,E)$: Specifically,  the \emph{layer-$r$ subgraph} at vertex $u \in V$, denoted by $G_{u; r}$, is the induced subgraph of $G$ over vertex set $\Delta_{u}^{r}:= \left\{v \in V: d_G(u,v) = r \right\}$, where $d_G$ is shortest-path distance in $G$. Viewing a graph as a 1-dimensional simplicial complex, we then aim to study the $1$st Betti number of such subgraphs. Our main result is that the $1$st Betti number of layer-$1$ subgraphs in Erd\H{o}s--R\'enyi random graphs $G(n,p)$ satisfies a central limit theorem.  
\end{abstract}

\maketitle

\section{Introduction}
\label{sec:intro}
The study of topological properties of random structures can be dated back to 1959, when Erd\H{o}s \cite{erdos1959graph} gave a probabilistic construction of a graph with large girth and large chromatic number. One variant of the construction is later called the \myER{} random graph $G(n,p)$, which is constructed by adding edges between all pairs of $n$ vertices with probability $p$ independently. Since then, many properties like the connectivity, largest components and the clique number in $G(n,p)$ have also been studied \cite{bollobas2001random}. Recently, due to the rapid development of topological data analysis (TDA) \cite{chazal2017introduction}, the homology of random simplicial complexes, which is a high-dimensional generalization of \myER{} random graphs, has received much attention \cite{kahle2013topology, kahle2014sharp, kahle2013limit}. 
In addition to the standard homological invariants, the newly developed persistent homology theory has also been applied to random simplicial complexes: For example, in \cite{bobrowski2014topology}, the authors consider the length of the longest barcode in persistence diagrams induced by some filtration consisting of random simplical complexes. 

In this article, we initiate the study of local topology of random graphs. We consider a type of local subgraphs called \emph{rooted $k$-neighborhood subgraphs}: for any graph $G = (V,E)$, the rooted $k$-neighborhood subgraph at vertex $u \in V$, denoted by $G_u^{k}$, is the induced subgraph of $G$ over vertex set $\Gamma_{u}^{k} =\left\{v \in V: d_G(u,v) \leq k \right\}$, where $d_G$ is the geodesic (shortest-path) distance. We take a first step to analyze the topology of rooted $k$-neighborhood subgraphs in \myER{} random graph $G(n,p)$. Specifically, we consider the ``layers'' of $G_u^{k}$ called \emph{layer-$r$ subgraphs}. 
\begin{definition}[Layer-$r$ subgraphs]
Given a graph $G = (V,E)$ with vertex set $V$ and edge set $E$, for any vertex $u \in V$, the \emph{layer-$r$ subgraph of $u$} is the induced subgraph over vertex set $\Delta_u^r = \left\{v \in V: d_G(u,v) = r \right\}$, where $d_G$ is the shortest-path metric. Denote such subgraph by $G_{u;r}$.
\end{definition}

A graph can be naturally viewed as a $1$-dimensional simplicial complex. Thus, for any graph $G = (V,E)$, we can define the \emph{$0$th Betti number $\beta_{0}(G)$} and \emph{$1$st Betti number $\beta_{1}(G)$}, which are the ranks of the $0$th and the $1$st homology group of the corresponding $1$-dimensional simplicial complex, respectively. Note that $\beta_{0}(G)$ is equal to the number of connected components in $G$ and $\beta_{1}(G) = |E| - |V| + \beta_{0}(G)$ by the Euler characteristic formula \cite{hatcher2002algebraic}. 

We are interest in the behavior of $\beta_{1}\left(G_{u;r}\right)$. In particular, we consider the following random variable $C_{n,p}^{(r)}$ defined by 
\begin{enumerate}
\item Sample a graph $G$ from $G(n,p)$;
\item Randomly pick a vertex $v$ in $G$; 
\item Set $C_{n,p}^{(r)} := \beta_{1}\left(G_{v;r}\right)$; Also set $c_{n,p}^{(r)} : = \beta_{0}\left(G_{v;r}\right)$.
\end{enumerate}

Throughout this paper, we use the standard Bachmann-Landau notation (asymptotic notation). That is, for real valued functions $f(n)$ and $g(n)$, as $n\rightarrow \infty$, we say
\begin{enumerate}
\item $f(n) = O(g(n))$: $\exists$ two constants $c > 0$ and $n_0 \in \mathbb{N}$ such that $|f(n)| \leq cg(n)$ for all $n \geq n_0$;
\item $f(n) = o(g(n))$: $\forall \epsilon > 0$, $\exists n_0 \in \mathbb{N}$ such that $|f(n)| < \epsilon g(n)$ for all $n \geq n_0$;
\end{enumerate}

We also use the notation $f \ll g$ to mean that $f(n)/g(n) \to 0$ as $n \to \infty$.

Our main result is that under some condition on $p$, $C_{n,p}^{(1)}$ satisfies a central limit theorem.
\begin{theorem}\label{thm:CLT_C}
If $n^{-1/2} \ll p < 1$, then
\begin{align*}
\frac{C_{n,p}^{(1)} - \myE[C_{n,p}^{(1)}]}{\sqrt{\myVar{C_{n,p}^{(1)}}}} \Rightarrow \mathcal{N}(0,1)
\end{align*} 
\end{theorem}

\vspace*{0.2in}

\paragraph{\bf Relation to persistent homology} For a given rooted $k$-neighborhood subgraph $G^k_u$, it can also be viewed as a metric graph equipped with the distance-to-root function $f_{G^k_u}$ \cite{dey2015comparing}: for any vertex $v$ in $G^k_u$, $f_{G^k_u}(v) = d_{G^k_u}(v,u)$; then we do linear interpolation for each edge. For example, suppose $z$ is the mid point of edge $(v,w)$, then $f_{G^k_u}(z) = \left[f_{G^k_u}(v) + f_{G^k_u}(w)\right] / 2$. 

Now consider the $1$-dimensional extended persistence diagram \cite{bendich2013homology} of $G^k_u$ induced by the super-level set filtration of the distance-to-root function. A special type of points in the diagram is the points on the diagonal. By an argument on extracting the \emph{Betti numbers} of some substructures from the extended persistence diagrams (Theorem 2 in \cite{bendich2013homology}), it is easy to see that the number of $(N,N)$ with any $N \in \mathbb{N}$ in the diagram associated with vertex $u$ is $\beta_1\left(G_{u;N}\right)$. Thus, Theorem \ref{thm:CLT_C} shows that the multiplicity of $(1,1)$ points in the $1$-dimensional extended persistence diagram satisfies a central limit theorem. See Appendix \ref{appendix:extended_persistence_diagram} for more details.

In the remainder of this paper, in Section \ref{sec:setting}, we will introduce background and some useful results that we will use later. We then prove the main theorem Theorem \ref{thm:CLT_C} in Section  \ref{sec:proof}.

\section{Definitions and useful lemmas}
\label{sec:setting}

\paragraph{\bf Weakly convergence and total variation distance} A sequence $\{X_n \}_{n=1}^{\infty}$ of random variables is said to \emph{converge weakly} to a limiting random variable $X$ (written $X_n \Rightarrow X$) if $\lim_{n \rightarrow \infty} \myE[f(X_n)] = \myE[f(X)]$ for all bounded continuous function $f$. The \emph{total variation distance} between real-valued random variables $X$ and $Y$ is defined by
\begin{align*}
d_{TV}(X, Y) := \sup_{A \in \mathcal{B}}\left| \myprob{X \in A} - \myprob{Y \in A} \right|
\end{align*}
where $\mathcal{B}$ is the class of Borel sets in $\mathbb{R}$. It also has the following equivalent form \cite{levin2017markov}:
\begin{align*}
d_{TV}(X, Y)= \frac{1}{2}\sup_f \left|\myE[f(X)] - \myE[f(Y)]\right|,
\end{align*}
with the supremum taken over all functions bounded by 1. Obviously, if $d_{TV}(X_n, X) \rightarrow 0$ as $n \rightarrow \infty$, then $X_n \Rightarrow X$. 

\paragraph{\bf Dissociated random variables} We say a set of random variables $\{\xi_{\bm{i}}: \bm{i} = (i_1, \cdots, i_r)\}_{\bm{i} \in I}$ for $I$ a set of unordered $r$-tuples is a set of \emph{dissociated random variables} if two subcollections of the random variables $\{\xi_{\bm{j}}\}_{\bm{j} \in J}$ and $\{\xi_{\bm{l}}\}_{\bm{l} \in L}$ are independent whenever $$\left(\bigcup_{\bm{j} \in J} \bm{j} \right) \bigcap \left(\bigcup_{\bm{l} \in L} \bm{l} \right)= \emptyset.$$ See \cite{silverman1976limit, moginley1975dissociated} for more details on dissociated random variables and their applications in random structures. Here, our main result heavily depends on the following normal approximation lemma on the sum of dissociated random variables.
\begin{lemma}[Stein's method, normal approximation \cite{barbour1989central}]\label{lem:normal_approx}
Suppose $\{\xi_{\bm{i}}\}_{\bm{i} \in I}$ is a set of dissociated random variables with $\myE[\xi_{\bm{i}}] = 0$ for each $i$. Set $W:=\sum_{\myi \in I}\xi_{\myi}$, and suppose $\myE[W^2] = 1$. For each $\myi \in I$, let $N_{\myi}:= \{\myj \in I: \myi \bigcap \myj \neq \emptyset \}$ be the dependency neighborhood for $\myi$. Let $Z = \mathcal{N}(0,1)$ be a standard normal random variable. Then, there is a universal constant $C$ such that 
\begin{align*}
d_{TV}(W, Z) \leq C \sum\limits_{\myi \in I}\sum\limits_{\myj, \myk \in N_{\myi}} \left\{\myE\left[ \left|\xi_{\myi} \xi_{\myj} \xi_{\myk}\right| \right] + \myE\left[\left|\xi_{\myi} \xi_{\myj} \right| \right] \myE\left[\left|\xi_{\myk} \right| \right] \right\}.
\end{align*}
\end{lemma}

The following well-known concentration inequality is also used in our proof.
\begin{lemma}[Chernoff bound \cite{dubhashi2009concentration}]\label{lem:chernoff}
Let $X_1, \cdots, X_n$ be independent random variables with values in $[0,1]$ and $X := \sum_{i}X_i$. Then, for any $\epsilon > 0$, we have
\begin{align*}
\myprob{X > (1 + \epsilon)\myE\left[X\right]} < \exp\left( -\epsilon^2 \myE\left[X \right] / 3 \right),\\
\myprob{X < (1 - \epsilon)\myE\left[X\right]} < \exp\left( -\epsilon^2 \myE\left[X \right] / 2 \right).
\end{align*}
\end{lemma}

In what follows, we often omit the parameters $n$ and $p$ from the notation $C_{n,p}^{(k)}$ and $c_{n,p}^{(k)}$ when their choices are clear from the context.

\section{Proof of Theorem \ref{thm:CLT_C}}
\label{sec:proof}

Recall that $C^{(1)}$ is the $1$st Betti number of a random layer-$1$ subgraph and $c^{(1)}$ is the $0$th Betti number (number of connected components) of the subgraph. First, we show that $C^{(1)} - c^{(1)}$ satisfies a central limit theorem for $n^{-1/2} \ll p < 1$. 

\begin{lemma}\label{lem:CLT_C-c}
Set $X^{(1)} := C^{(1)} - c^{(1)}$. Let $Z = \mathcal{N}(0,1)$ be a standard normal random variable. If $n^{-1/2} \ll p < 1$, then
\begin{align*}
d_{TV} \left(\frac{X^{(1)} - \myE[X^{(1)}]}{\sqrt{\myVar{X^{(1)}}}}, Z\right) = o(1)
\end{align*} 
\end{lemma}

\begin{proof}
It is easy to see that the number of neighbors of an individual vertex of $G(n,p)$ is a Binomial random variable with parameters $(n-1)$ and $p$. Denote the sampled graph by $G = (V,E)$. We pick an arbitrary vertex $u \in V$ and consider the layer-$1$ subgraph $G_{u;1}$ at $u$. Let $\alpha$ be the number of vertices in $G_{u;1}$ and $\gamma$ be the number of edges in $G_{u;1}$. Then, we know that $X^{(1)} = C^{(1)} - c^{(1)} = \gamma - \alpha$. Let $V' = V\setminus \{u\}$ be the set of vertices excluding $u$. Without loss of generality, we assume the vertex set $V' = \{v_1,v_2,\cdots, v_{n-1}\}$. For any $(s,t) \in I := \left\{(v_i, v_j): 1 \leq i \leq j \leq n-1 \right\}$, the corresponding indicator random variable $\eta_{(s,t)}$ is defined as follows:
\begin{align*}
\eta_{(s,t)} = \left\{
\begin{array}{cl}
1 , & \text{ $\{s\} \cup \{t\} \cup \{u\}$ spans a complete graph;}\\
0, & \text{ Otherwise}\\
\end{array}
\right.
\end{align*}

It is easy to see that the following holds.
\begin{align*}
X^{(1)} = \gamma - \alpha = \sum\limits_{(s,t) \in I}(-1)^{\mathbbm{1}_{\{s = t\}}}\cdot \eta_{(s,t)}
\end{align*} 
where $\mathbbm{1}_{\{s = t\}}$ is the indicator function.

Set $\sigma^2 := \myVar{X^{(1)}}$. Define a finite collection of random variables $\left\{\xi_{(s,t)}\right\}_{(s,t) \in I}$ with
\begin{align*}
\xi_{(s,t)} := \frac{(-1)^{\mathbbm{1}_{\{s = t\}}}}{\sigma}\left(\eta_{(s,t)} - \myE\left[ \eta_{(s,t)} \right]\right)
\end{align*}
as the ``normalized'' version of $\left\{\eta_{(s,t)}\right\}$: it is easy to check $\myE\left[\xi_{(s,t)}\right] = 0$. Now set 
\begin{align}
W := \sum_{(s,t) \in I} \xi_{(s,t)} = \frac{X^{(1)} - \myE\left[X^{(1)} \right]}{\sqrt{\myVar{X^{(1)}}}}, \label{eqn:Wdef}
\end{align}
where the last equality holds due to the linearity of expectation. Thus we have $\myE\left[W^2\right] = 1$. Decompose the collection by $I = I_v \sqcup I_e$ where $I_v : = \left\{(v_i, v_i): v_i \in V' \right\}$ and $I_e := \left\{(v_i, v_j) \in V' \times V': v_i \neq v_j \right\}$. It is not hard to check that $\left\{\xi_{(s,t)} \right\}_{(s,t) \in I}$ is a collection of dissociated random variables. 

Note that for any $2$-tuples $A, B\in I$, we have 
\begin{align*}
\myE\left[\eta_A \right] \myE\left[\eta_B \right] &= \myprob{\eta_A = 1} \myprob{\eta_B = 1} \\
&\leq \myprob{\eta_A = 1 \middle| \eta_B = 1} \myprob{\eta_B = 1} = \myprob{\eta_A \eta_B = 1} = \myE\left[\eta_A \eta_B\right].
\end{align*}
Thus, for any $2$-tuples $A, B, C \in I$, by expanding the expectation (16 terms in all), it is not hard to see
\begin{align*}
\myE\left[ \left|\xi_{A} \xi_{B} \xi_{C}\right| \right] + \myE\left[\left|\xi_{A} \xi_{B} \right| \right] \myE\left[\left|\xi_{C} \right| \right] \leq \frac{16}{\sigma^3} \myE\left[\eta_A \eta_B \eta_C \right]
\end{align*}
To prove Lemma \ref{lem:CLT_C-c}, by Lemma \ref{lem:normal_approx} and combining Eqn. \ref{eqn:Wdef}, it suffices to prove
\begin{align*}
\frac{1}{\sigma^3} \sum\limits_{A \in I} \sum\limits_{B,C \in N_{A}}\myE\left[\eta_A \eta_B \eta_C \right] \rightarrow 0 \text{ as } n \rightarrow \infty
\end{align*}
where $N_A$ is the subcollection of $2$-tuples of $I$ sharing one (vertex) element with $A$. We can further decompose the summation as follows.
\begin{align}
\sum\limits_{A \in I} \sum\limits_{B,C \in N_{A}}\myE\left[\eta_A \eta_B \eta_C \right] &= \sum\limits_{A \in I_v} \sum\limits_{B,C \in N_{A}}\myE\left[\eta_A \eta_B \eta_C \right] + \sum\limits_{A \in I_e} \sum\limits_{B,C \in N_{A}}\myE\left[\eta_A \eta_B \eta_C \right]\nonumber
\end{align}
Note that by enumerating all the $2$-tuples ($(n-2)$ tuples) in dependency neighborhoods of $I_v$, we have
\begin{align}\label{eqn:3terms_exp_first_part}
\sum\limits_{A \in I_v} \sum\limits_{B,C \in N_{A}}\myE\left[\eta_A \eta_B \eta_C \right] &= \sum\limits_{A \in I_v} \sum\limits_{B,C \in N_{A}}\myprob{\eta_B \eta_C = 1 \middle| \eta_A = 1} \myprob{\eta_A = 1} \nonumber\\
&= \sum\limits_{A \in I_v}  \myprob{\eta_A = 1} \sum\limits_{B,C \in N_{A}}\myprob{\eta_B \eta_C = 1 \middle| \eta_A = 1}\nonumber\\
&= \binom{n-1}{1} \myprob{\eta_{\left(v_1,v_1\right)} = 1} \sum\limits_{B,C \in N_{\left(v_1,v_1\right)}}\myprob{\eta_B \eta_C = 1 \middle| \eta_{\left(v_1,v_1\right)} = 1}\nonumber\\
& = (n-1)p \left[ (n-2)p^2 + (n-2)(n-1)p^4 \right]
\end{align} 
Similarly, we have
\begin{align}\label{eqn:3terms_exp_second_part}
&\sum\limits_{A \in I_e} \sum\limits_{B,C \in N_{A}}\myE\left[\eta_A \eta_B \eta_C \right]\nonumber\\
&= \binom{n-1}{2} \sum\limits_{B,C \in N_{\left(v_1,v_2\right)}}\myE\left[\eta_{\left(v_1,v_2\right)} \eta_B \eta_C \right]\nonumber\\
&= \binom{n-1}{2} \left\{\left(\sum\limits_{\substack{ B,C \in N_{\left(v_1,v_2\right)}\\ B, C \in I_v}} + \sum\limits_{\substack{ B,C \in N_{\left(v_1,v_2\right)}\\ B, C \in I_e}} + 2\sum\limits_{\substack{ B,C \in N_{\left(v_1,v_2\right)}\\ B \in I_v, C \in I_e}}\right) \myE\left[\eta_{\left(v_1,v_2\right)} \eta_B \eta_C \right]\right\}\nonumber\\
&= \binom{n-1}{2} \left\{4p^3 + 2(n-3)\left[p^5 + p^6 + 2(n-4)p^7 \right] + 2 \left[4(n-3)p^5 \right]\right\}
\end{align} 
Combining Eqn. \ref{eqn:3terms_exp_first_part} and Eqn. \ref{eqn:3terms_exp_second_part}, we know the following inequality holds. 
\begin{align}\label{eqn:3terms_exp_final}
\sum\limits_{A \in I} \sum\limits_{B,C \in N_{A}}\myE\left[\eta_A \eta_B \eta_C \right] < 3n^2p^3 + 6n^3p^5 + n^3p^6 + 2n^4p^7
\end{align}
Since $p \gg n^{-1/2}$, the leading order of the right hand side of Eqn. \ref{eqn:3terms_exp_final} is $n^4p^7$.

In what follows, we calculate $\sigma$ the standard deviation of $X^{(1)}$. First, note that 
\begin{align}\label{eqn:sigma_main}
\sigma^2 = \myVar{X^{(1)}} = \myVar{\gamma - \alpha} = \myVar{\alpha} + \myVar{\gamma} - 2 \cdot \mathrm{Cov}\left(\alpha, \gamma \right).
\end{align}
Also note that $\alpha = \sum_{A \in I_v} \eta_A$ and $\gamma = \sum_{A \in I_e} \eta_A$. By the linearity of expectation, we have $\myE\left[\alpha\right] = (n-1)p$ and $\myE\left[\gamma\right] = \binom{n-1}{2}p^3$. Note that $\alpha$ is the sum of $(n-1)$ independent random variable, thus $\myVar{\alpha} = (n-1)p(1-p)$. Then, by enumerating $A, B \in I_e$ by the size $i$ of their intersection, we know that
\begin{align}\label{eqn:var_beta}
\myVar{\gamma} &= \left\{\sum\limits_{A,B \in I_e} \myE\left[\eta_A \cdot \eta_B \right] \right\}- \left\{\myE\left[ \gamma \right] \right\}^2 \nonumber\\
&= \binom{n-1}{2} p^3 \left[\binom{n-3}{2}p^3 + \binom{2}{1}\binom{n-3}{1} p^2 + 1 \right] - \left[ \binom{n-1}{2} p^3\right]^2
\end{align}

Similarly, we can expand the covariance as follows.
\begin{align}\label{eqn:cov_alpha_beta}
\mathrm{Cov}\left(\alpha, \gamma \right) &= \myE\left[\sum\limits_{A \in I_v} \eta_A \cdot \sum\limits_{B \in I_e} \eta_B   \right] - \myE\left[ \alpha\right] \cdot \myE\left[\gamma \right] \nonumber\\
&= \binom{n-1}{1} p \left[\binom{n-2}{2}p^3 + \binom{n-2}{1}p^2 \right] - (n-1)p \cdot \binom{n-1}{2}p^3\nonumber\\
&= (n-1)(n-2)p^3(1-p) 
\end{align}

Finally, plugging Eqn. \ref{eqn:var_beta} and Eqn. \ref{eqn:cov_alpha_beta} back into Eqn. \ref{eqn:sigma_main} and by routine calculations, we have
\begin{align}\label{eqn:sigma_final}
\sigma^2 = (n-1) \bigg[p - &p^2 -(3/2)(n-2)p^3 + 2(n-2)p^4 \nonumber\\
 &+ (n-2)(n-3)p^5 - (1/2)(n-2)(2n-5)p^6\bigg]
\end{align}
Note that if $p < 1$, then by Eqn. \ref{eqn:sigma_final}, we know that $\sigma^2 \neq 0$. Also it is not hard to see that $p \gg n^{-1/2}$ implies that the leading order of $\sigma^2$ is $n^3p^5$. Finally, combining this fact with Eqn. \ref{eqn:3terms_exp_final}, we know that there exists some constant $C$ such that
\begin{align*}
\frac{1}{\sigma^3} \sum\limits_{A \in I} \sum\limits_{B,C \in N_{A}}\myE\left[\eta_A \eta_B \eta_C \right] < C\left(\frac{n^4p^7}{n^{9/2}p^{15/2}}\right) = \frac{C}{(np)^{1/2}} \rightarrow 0
\end{align*}
\end{proof}

Next, we show the follow lemma on $c^{(1)}$, which intuitively says that if $p \gg n^{-1/2}$, then with high probability, $c^{(1)} = 1$.

\begin{lemma}\label{lem:estimate_c1}
If $p \gg n^{-1/2}$, then $\myprob{c^{(1)} > 1} = o (1)$.
\end{lemma}

\begin{proof}
Denote the number of vertices in $G_{u;1}$ by $\alpha$, which is a random variable with binomial distribution $Bin(n-1, p)$. Thus, by using Chernoff bound (Lemma \ref{lem:chernoff}), we have
\begin{align*}
\myprob{\alpha > 2(n-1)p} < \exp\left( - (n-1)p / 3 \right),\\
\myprob{\alpha < (1/2) (n-1)p} < \exp\left( - (n-1)p / 8 \right).
\end{align*}

By applying the law of total probability, we know that
\begin{align}\label{eqn:c_1_main}
&\myprob{c^{1} > 1} \nonumber\\
=& \myprob{G_{u;1} ~\text{is not connected}} \nonumber\\
\leq& \myprob{\alpha > 2(n-1)p} + \myprob{\alpha < (1/2)(n-1)p} \nonumber\\
&+ \myprob{G_{u;1} ~\text{is not connected} \middle| (1/2)(n-1)p \leq \alpha \leq 2(n-1)p} \nonumber\\
 <& \exp\left( - (n-1)p /3 \right) + \exp\left( - (n-1)p / 8 \right) \\
 &+ \myprob{G_{u;1} ~\text{is not connected} \middle| (1/2)(n-1)p\leq \alpha \leq 2(n-1)p} \nonumber
\end{align}

For a fixed $k \in \left[(1/2)(n-1)p, 2(n-1)p\right]$, let $Y_i^{(k)}$ be the number of components with $i$ vertices in $G_{u;1}$ conditioned on $\alpha = k$. Thus, we have
\begin{align}\label{eqn:c_1_conditioned}
\myprob{G_{u;1} ~\text{is not connected} \middle| \alpha = k} = \myprob{\bigcup\limits_{i=1}^{\floor*{k/2}}\{Y_i^{(k)} > 0 \}} \leq \sum_{i=1}^{\floor*{k/2}}\myprob{Y_i^{(k)} > 0}
\end{align}
Note that since $p \gg n^{-1/2}$, $k$ goes to infinity as $n$ goes to infinity. By applying Markov's inequality and noticing the fact that for any $x \in [0,1]$, $1-x \leq e^{-x}$, we know that
\begin{align}\label{eqn:isolated_points}
\myprob{Y_1^{(k)} > 0} \leq \myE\left[ Y_1^{(k)}\right] = k(1-p)^{k-1} < k(1-p)^{k/2} \leq k e^{-pk/2} < 2npe^{-np^2/8}
\end{align}
Similarly, since the number of spanning trees on a fixed set of $i$ vertices is $i^{i-2}$ (the so-called Cayley's formula), we have
\begin{align}\label{eqn:connectedcomponents}
\sum\limits_{i=2}^{\floor*{k/2}}\myprob{Y_i^{(k)} > 0} \leq \sum\limits_{i=2}^{\floor*{k/2}}\myE\left[ Y_i^{(k)}\right] &\leq \sum_{i=2}^{\floor*{k/2}}\binom{k}{i}i^{i-2}(1-p)^{i(k-i)}p^{i-1}\nonumber\\
& \leq \sum_{i=2}^{\floor*{k/2}} \left(\frac{ek}{i} \right)^{i} i ^{i-2}(1-p)^{i(k/2)}p^{i-1}\nonumber\\
& = \sum_{i=2}^{\floor*{k/2}}\frac{1}{pi^2}\left(ekp(1-p)^{k/2} \right)^{i}\nonumber\\
& < \frac{1}{4p} \sum_{i=2}^{\floor*{k/2}}\left(ekp(1-p)^{k/2} \right)^{i} \nonumber\\
& = \frac{1}{4p} \left(ekp(1-p)^{k/2} \right)^2 \frac{1- \left(ekp(1-p)^{k/2} \right)^{\floor*{k/2} -1}}{1- ekp(1-p)^{k/2} }
\end{align}
Note that 
\begin{align*}
ekp(1-p)^{k/2} \leq ekpe^{-pk/2} < 6 (np)e^{-np^2/8} = 6 e^{-np^2/8 + \ln (np)}
\end{align*}
Also note that $p \gg n^{-1/2}$ implies $np^2/8 >  \ln (np)$. Thus, we know that $ekp(1-p)^{k/2} \rightarrow 0$ as $n \rightarrow \infty$. Furthermore, applying this fact to Eqn. \ref{eqn:connectedcomponents}, for large enough $n$, we have
\begin{align}\label{eqn:connectedcomponents_final}
\sum\limits_{i=2}^{\floor*{k/2}}\myprob{Y_i^{(k)} > 0} < \frac{1}{p}\left(ekp(1-p)^{k/2} \right)^2 \leq 36n^2pe^{-np^2/4}
\end{align}
Combining Eqn. \ref{eqn:c_1_conditioned} with Eqn. \ref{eqn:connectedcomponents_final} and Eqn. \ref{eqn:isolated_points}, we know that for any fixed $k \in \left[(1/2)(n-1)p, 2(n-1)p\right]$, we have
\begin{align*}
\myprob{G_{u;1} ~\text{is not connected} \middle| \alpha = k} < 2npe^{-np^2/8} + 36n^2pe^{-np^2/4}
\end{align*}
Thus, we have the following estimate for the last term on the right side of Eqn. \ref{eqn:c_1_main}.
\begin{align*}
&\myprob{G_{u;1} ~\text{is not connected} \middle| (1/2)(n-1)p\leq \alpha \leq 2(n-1)p}\\
<& \frac{3}{2}np\left( 2npe^{-np^2/8} + 36n^2pe^{-np^2/4}\right)\\
=& 3 e^{-np^2/8 + 2\log (np)} + 54e^{-np^2/4 + \log(np^2) + \log(np)} = o(1)
\end{align*} 
Plugging the above equation back to Eqn. \ref{eqn:c_1_main} concludes the proof of Lemma \ref{lem:estimate_c1}.
\end{proof}

\begin{proof}[\bf Proof of Theorem \ref{thm:CLT_C}]
Let $Z = \mathcal{N}(0,1)$ be a standard normal random variable. Set
\begin{align*}
\widetilde{X^{(1)}} := \frac{X^{(1)} - \myE[X^{(1)}]}{\sqrt{\myVar{X^{(1)}}}}, ~~~\widetilde{C^{(1)}} := \frac{C^{(1)} - \myE[C^{(1)}]}{\sqrt{\myVar{C^{(1)}}}}.
\end{align*}
To prove the theorem, it suffices to show that $d_{TV}\left(\widetilde{C^{(1)}},  Z\right) = o(1)$. 
Let $\mathcal{B}$ denote the class of Borel sets in $\mathbb{R}$. Recall that $X^{(1)} = C^{(1)} - c^{(1)}$. And it is easy to check that under the assumption $c^{(1)} = 1$, we have $\widetilde{X^{(1)}} = \widetilde{C^{(1)}}$. Note that for any Borel set $A \in \mathcal{B}$, we have
\begin{align*}
\myprob{\widetilde{X^{(1)}} \in A} - \myprob{\widetilde{C^{(1)}} \in A}  &\leq \myprob{\widetilde{X^{(1)}} \in A \middle| c^{(1)} = 1} + \myprob{c^{(1)} > 1}  - \myprob{\widetilde{C^{(1)}} \in A} \\
&= \myprob{\widetilde{C^{(1)}} \in A} + \myprob{c^{(1)} > 1}  - \myprob{\widetilde{C^{(1)}} \in A}\\
&= \myprob{c^{(1)} > 1} 
\end{align*}
On the other hand, we can also derive a lower bound. 
\begin{align*}
\myprob{\widetilde{X^{(1)}} \in A} - \myprob{\widetilde{C^{(1)}} \in A}  &\geq \myprob{\widetilde{X^{(1)}} \in A \middle| c^{(1)} = 1} \myprob{c^{(1)} = 1} - \myprob{\widetilde{C^{(1)}} \in A}\\
& = - \myprob{\widetilde{C^{(1)}} \in A} \left(1 -  \myprob{c^{(1)} = 1}\right)\\
& \geq - \myprob{c^{(1)} > 1} 
\end{align*}
Thus, we have
\begin{align*}
d_{TV}\left(\widetilde{C^{(1)}}, \widetilde{X^{(1)}} \right) = \sup\limits_{A \in \mathcal{B}} \left|\myprob{\widetilde{C^{(1)}} \in A} - \myprob{\widetilde{X^{(1)}} \in A} \right| \leq \myprob{c^{(1)} > 1} 
\end{align*}
By triangle inequality, we know
\begin{align*}
d_{TV}\left(\widetilde{C^{(1)}}, Z\right) \leq d_{TV}\left(\widetilde{C^{(1)}}, \widetilde{X^{(1)}} \right) + d_{TV}\left(\widetilde{X^{(1)}}, Z\right) \leq \myprob{c^{(1)} > 1} + d_{TV}\left(\widetilde{X^{(1)}}, Z\right)
\end{align*}
Finally, applying Lemma \ref{lem:CLT_C-c} and Lemma \ref{lem:estimate_c1} concludes the proof.
\end{proof}

\section{Comments and open questions}
In this paper, we discussed the behavior of $C_{n,p}^{(1)}$ of \myER{} random graphs $G(n,p)$, which is a first step to understand the local topology of random graphs. However, there are still some unsolved questions related to the $1$st Betti number of layer-$r$ subgraphs. We assume the sampled graph is $G = (V,E)$.

\begin{enumerate}
\item[1.] Let $\beta_{1,d}:= \left|\left\{ v \in V: \beta_1\left(G_{v;1} \right) = d \right\} \right|$. Inspired by the standard results on the degree distribution of random graphs (Section 3.1 in \cite{frieze2016introduction}), a natural question arises: what is the distribution of $\beta_{1,d}$? Based on our empirical result on the $p$-values of the D'Agostino-Pearson test\footnote{The D'Agostino-Pearson test is a normality test with the null hypothesis being ``the samples are normally distributed''. A small $p$-value (typically when $\leq .05$) indicates strong evidence against the null hypothesis, which means that with high probability, the samples are not normally sampled. A large $p$-value (typically when $> .05$) indicates weak evidence against the null hypothesis, which means we fail to reject the null hypothesis.} \cite{d1973tests} (see Figure \ref{fig:distribution_betti1_1_ring}), we conjecture that when $p$ is large enough, then $\beta_{1,d}$ should obey a normal distribution.
 
\begin{figure}[h]
  \centering
  \begin{tabular}{c}
  \includegraphics[height=5cm]{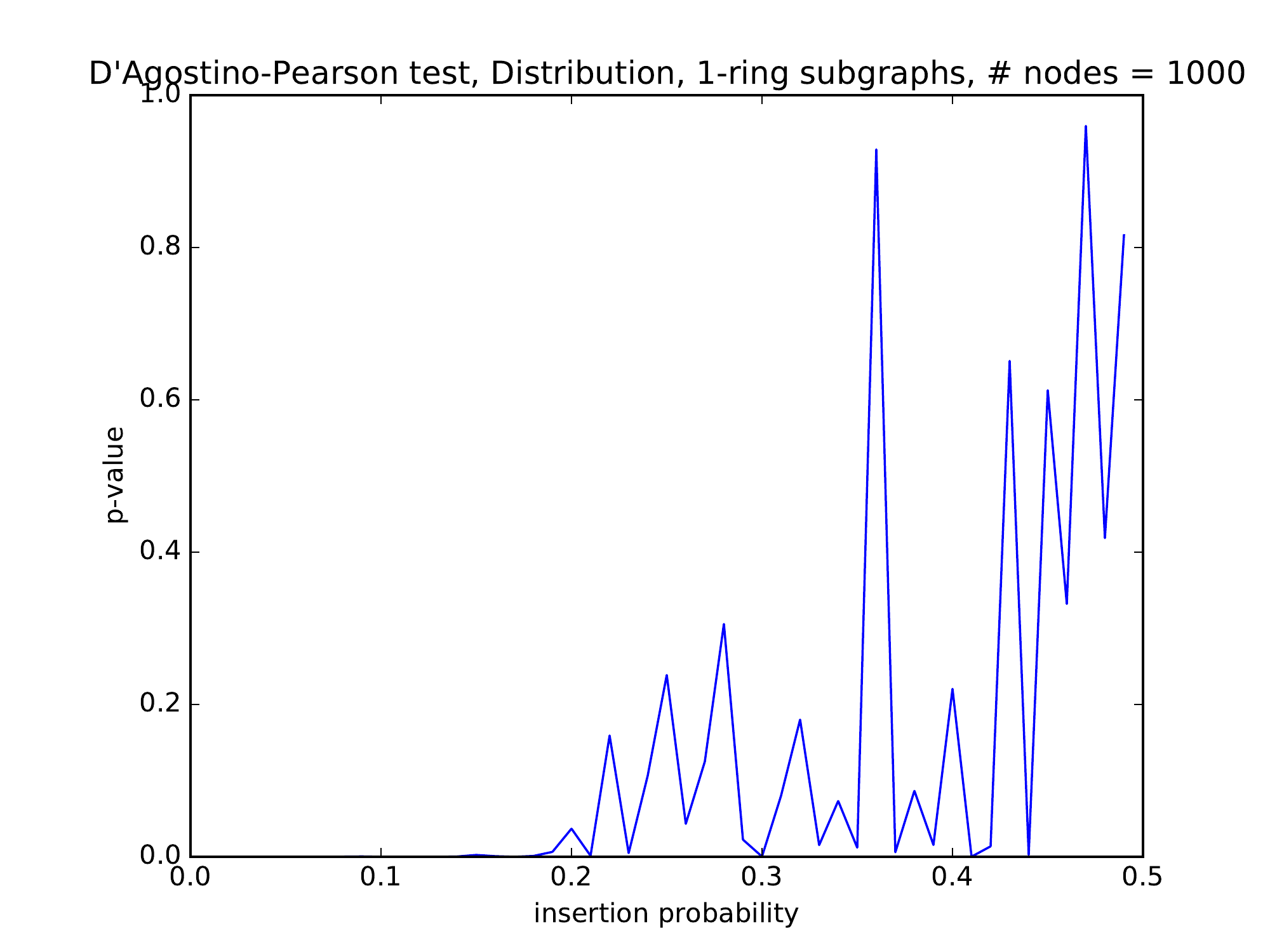}  \\
(a) $G(1000, p)$ with $p \in [0, .5]$
  \end{tabular}
\vspace*{-0.15in}
  \caption{Given a graph $G$ sampled from \myER{} random graph $G(1000,p)$, we perform the so-called D'Agostino-Pearson test on the 1000 samples $\left\{\beta_{1}(G_{v;1}): v \in G \right\}$ (one value for each vertex). We simply use the function \textsc{scipy.stats.normaltest} in Python to compute the $p$-values of these tests (y-axis). Note that a large $p$-value doesn't mean that the samples are actually sampled from a normal distribution. However, the result still gives us a hint on how large $p$ should be such that the distribution of $\beta_{1,d}(G)$ looks like a normal distribution.}
  \label{fig:distribution_betti1_1_ring}
\end{figure}

\item[2.] We only consider the layer-$1$ subgraphs in $G(n,p)$. So how about layer-$k$ subgraphs? Interestingly, for $k=2,3$, our empirical result indicates that there should exist two disjoint ranges of $p$ such that when $p$ falls in either range, a central limit theorem of $C_{n,p}^{(k)}$ holds (see Figure \ref{fig:betti1_2_3_ring}). 

\end{enumerate}

\begin{figure}[t]
  \centering
  \begin{tabular}{ccc}
  \includegraphics[height=5cm]{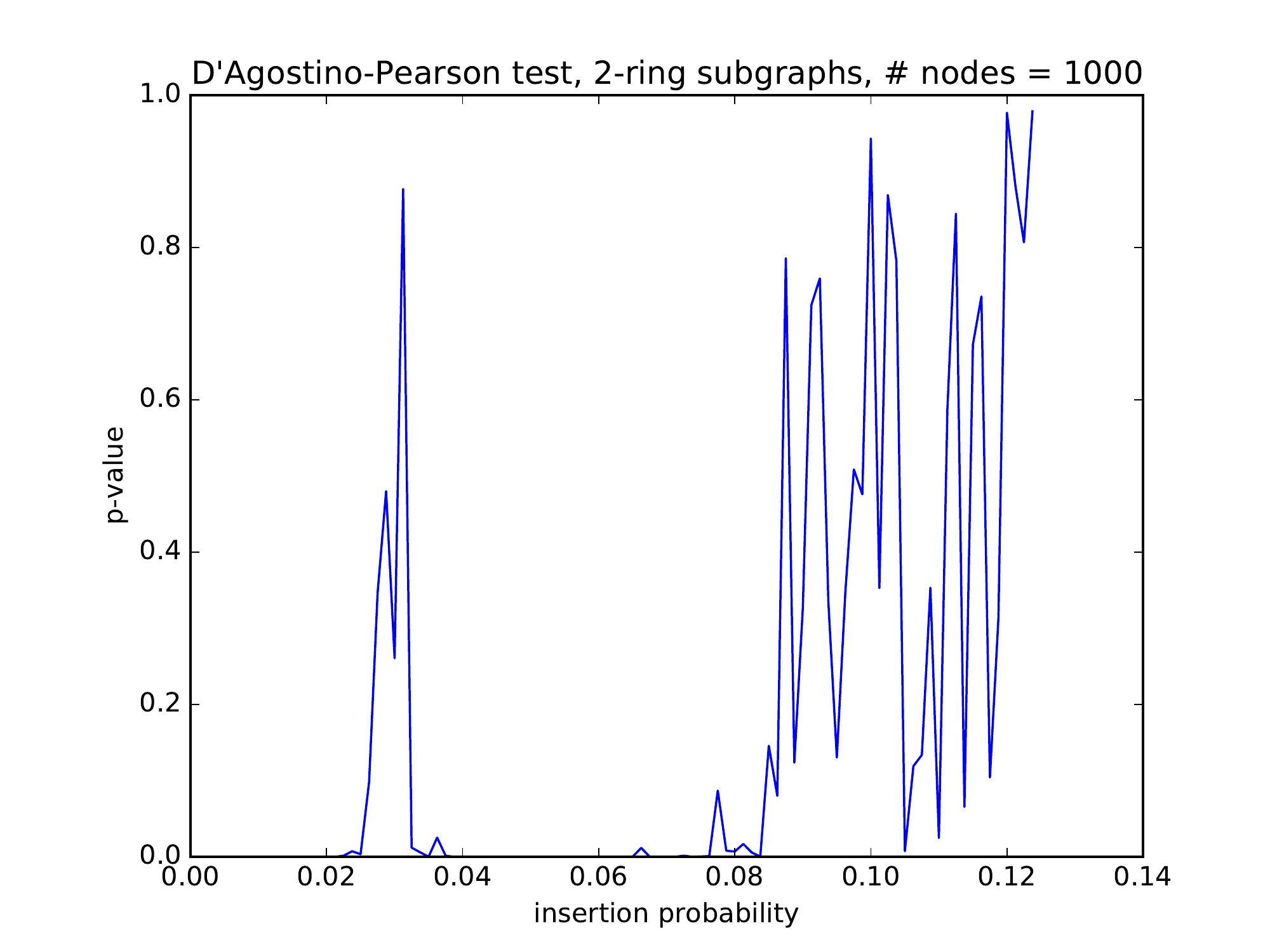} &  & \includegraphics[height=5cm]{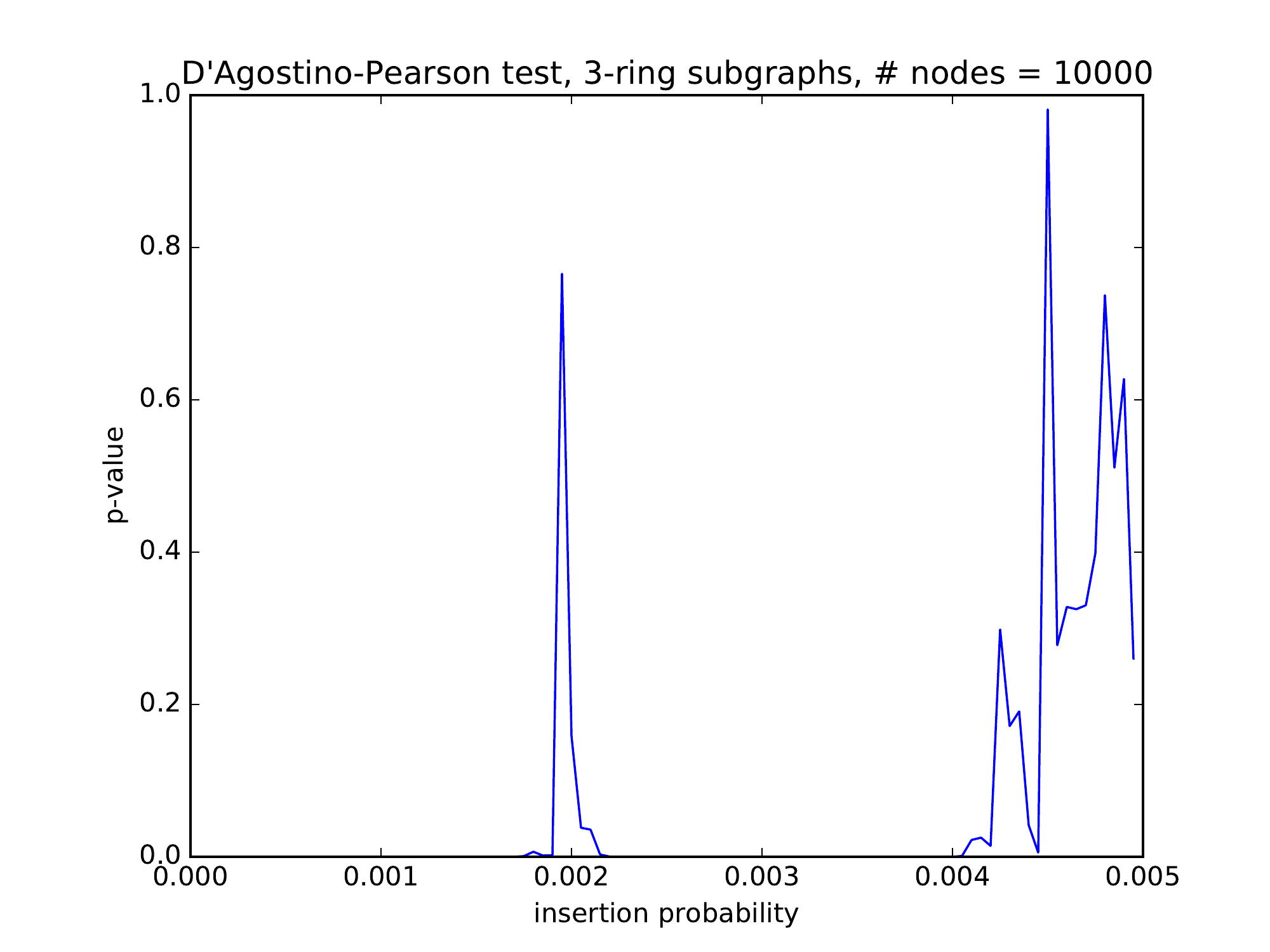} \\
(a) $G(1000, p)$ with $p \in [0, .125]$ & & (b) $G(10000, p)$ with $p \in [0, .005]$
  \end{tabular}
\vspace*{-0.15in}
  \caption{Again, we perform the D'Agostino-Pearson tests on 1000 samples of $C_{1000,p}^{(2)}$ and $C_{10000,p}^{(3)}$, respectively. These results don't directly support our conjecture, but they still give us a hint on the range of $p$ in which a central limit theorem of the corresponding $1$st Betti number may hold.}
  \label{fig:betti1_2_3_ring}
\end{figure}

\bibliographystyle{amsalpha}
\bibliography{cyclomatic}

\appendix
\section{Information encoded in the extended persistence diagrams}\label{appendix:extended_persistence_diagram}

Let $G = (V, E)$ be a connected unweighted finite graph with at least two vertices. Fix an arbitrary node $v \in V$ and let $f_v: V \rightarrow \mathbb{Z}_{\geq 0}$ be the graph distance function (i.e. for $\forall u \in V$, $f_v(u)$ is the shortest-path distance from $u$ to $v$; $f_v(v)=0$). Set $h_v := \max_{u \in V} f_v(u)$ (the height of the shortest-path tree). Recall that the layer-$k$ subgraph $G_{v;k}$ is just the induced subgraph on vertex set $\{u \in V \mid f_v(u)=k\}$.

Note that $G$ can also be viewed a finite $1$-dimensional CW complex $X$ \cite{hatcher2002algebraic} and thus a compact topological space. We now define a piecewise linear function $F_v: X \rightarrow \mathbb{R}$ such that $F_v$ and $f_v$ agree at $X^{(0)}$ (the $0$-skeleton of $X$). That is, for every point $x \in X$, if $x \in X^{(0)}$, then set $F_v(x) = f_v(x)$; otherwise, there exists a $1-$cell $e_{\alpha}^1 \ni x$, then set $F_v(x) = g_{\alpha}\left(\phi_{\alpha}^{-1}(x)\right)$, where $\phi_{\alpha}$ is the corresponding characteristic map of $e^1_{\alpha}$ and
\begin{align*}
g_{\alpha}(t) := \left[f_v\left(\rho^1_{\alpha}(0)\right) - f_v\left(\rho^1_{\alpha}(1)\right)\right]t + f_v\left(\rho^1_{\alpha}(0)\right).
\end{align*}

Follow the idea of \cite{bendich2013homology}, we consider the \emph{super-level set filtration} of $X$ constructed by the \emph{super-level sets} $X^t(F_v) = F_v^{-1}[t, +\infty)$, for all real values $t$. We use the direct sum $\mathsf{H}(X^t(F_v)) := \mathsf{H}_{0}(X^t(F_v)) \oplus \mathsf{H}_{1}(X^t(F_v))$ to suppress the homological dimension.  We also borrow the definitions of \emph{homological regular value} and \emph{homological critical value} from \cite{bendich2013homology}: A real value $t$ is called a homological regular value of $f$ if there exists $\epsilon > 0$ such that the map between homology groups induced by the inclusion $X^{t+\delta}(F_v) \xhookrightarrow{} X^{t-\delta}(F_v)$ is an isomorphism for every $\delta < \epsilon$; Otherwise, $t$ is called a homological critical value. We now artificially add $(h_v + 2)$ homological regular values $s_0, s_1, \cdots, s_{h_v + 1}$ to the sequence with $s_i = h_v - i + 0.5$. Set $\mathcal{X}^i := \mathcal{X}^{s_i}(F_v) = F_v^{-1}[s_i, +\infty)$ and $\mathcal{X}_i := F_v^{-1}\left(-\infty, s_{(h_v + 1) -i}\right]$. We construct the following \emph{extended filtration} \cite{bendich2013homology}:
\begin{align}\label{diag:ex_filtration}
0 = \mathsf{H}\left(\mathcal{X}^{0}\right) \rightarrow \cdots \rightarrow \mathsf{H}\left(\mathcal{X}^{h_v + 1}\right) = \mathsf{H}\left(\mathcal{X}, \mathcal{X}_{0}\right) \rightarrow \cdots \rightarrow     \mathsf{H}\left(\mathcal{X}, \mathcal{X}_{h_v + 1}\right) = 0
\end{align}

Suppose all the homological critical values are $t_1 > t_2 > \cdots > t_m = 0$. Note that for any two consecutive homological regular values $s_i$ and $s_{i+1}$, there is at most one homological critical value $r_j$ such that $s_i > t_j > s_{i+1}$. Also for any homological critical value $t_j$, it must be between two consecutive homological regular values. Thus, we can define an injection $d: [m] \rightarrow [2h_v+1]$ such that $s_{d(i) -1} > t_i > s_{d(i)}$.

The corresponding ($0$-dim and $1$-dim) persistence diagrams (represented in $t_i$ after applying $d^{-1}$) \cite{edelsbrunner2010computational} of the filtration \ref{diag:ex_filtration} are called the ($0$th and $1$st) \emph{extended persistence diagrams} of the super-level set filtration induced by $F_v$. We denote the $0$th and $1$st extended persistence diagrams as $D_0\left(F_v; G\right)$ and $D_1\left(F_v; G\right)$, respectively. In what follows, we will drop the notation $F_v$ for simplicity.

By directly applying a variant of Theorem 2 in \cite{bendich2013homology} (they consider the sub-level set filtration, but here we consider the super-level set filtration), we can extract the $0$th and $1$st Betti numbers of $G_{v;k}$ from the $0$th and the $1$st extended persistence diagrams of the super-level set filtration induced by $F_v$. We use $|\cdot |$ to denote the number of elements in a multiset. 

\begin{claim}\label{claim:betti1}
For any $k \geq 1$, we have 
\begin{align*}
\beta_0\left(G_{v;k}\right) &= \left|\left\{(x, y) \in D_0: x \geq k, y \leq k-1 \right\}\right| + \left|\left\{(x, y) \in D_1: x\leq k-1, y\geq k+1 \right\}\right| \\
\beta_1\left(G_{v;k}\right) &= \left|\left\{(k,k) \in D_1 \right\}\right|
\end{align*}
\end{claim}

Recall that any connected graph $G$ has a stratified structure after introducing the graph distance function $f_v$ at a given root $v$. That is, $G$ can be viewed as a collection of layer-$k$ subgraphs $\{G_{v;k}\}_{0\leq k \leq h_v}$ together with the links between two consecutive layers. Those links can be defined formally as follows. 
\begin{definition}
An edge $e \in E$ is called a $(k, k+1)$-crossing if it connects a vertex in $G_{v;k}$ to a vertex in $G_{v;(k+1)}$. 
\end{definition}

Besides the $0$th and the $1$st Betti numbers, we also aim to recovery the following three basic quantities from $D_0$ and $D_1$.
\begin{itemize}
\denselist
\item[(a)] $\sigma_k :=$ Number of vertices in $G_{v;k}$ for each $k$;
\item[(b)] $\epsilon_k : =$ Number of edges in $G_{v;k}$ for each $k$;
\item[(c)] $\omega_k :=$ Number of $(k, k+1)$-crossings for each $k \leq h_v-1$;
\end{itemize}
From these three quantities, one can depict a very brief ``shape'' of the shortest-path tree. An observation is that we can actually recover $\omega_k$ from $D_0$ and $D_1$.
\begin{claim}
For any $k \geq 0$, we have $$\omega_k = \left|\left\{(x,y) \in D_0: x \geq k+1, y\leq k \right\}\right| + \left|\left\{(x,y) \in D_1: x \leq k, y \geq k+1 \right\}\right|.$$
\end{claim}

The proof is simply combining Euler's characteristic formula and Claim \ref{claim:betti1}, thus is omitted.

\begin{figure}[t]
  \centering
  \begin{tabular}{ccc}
  \includegraphics[height=4cm]{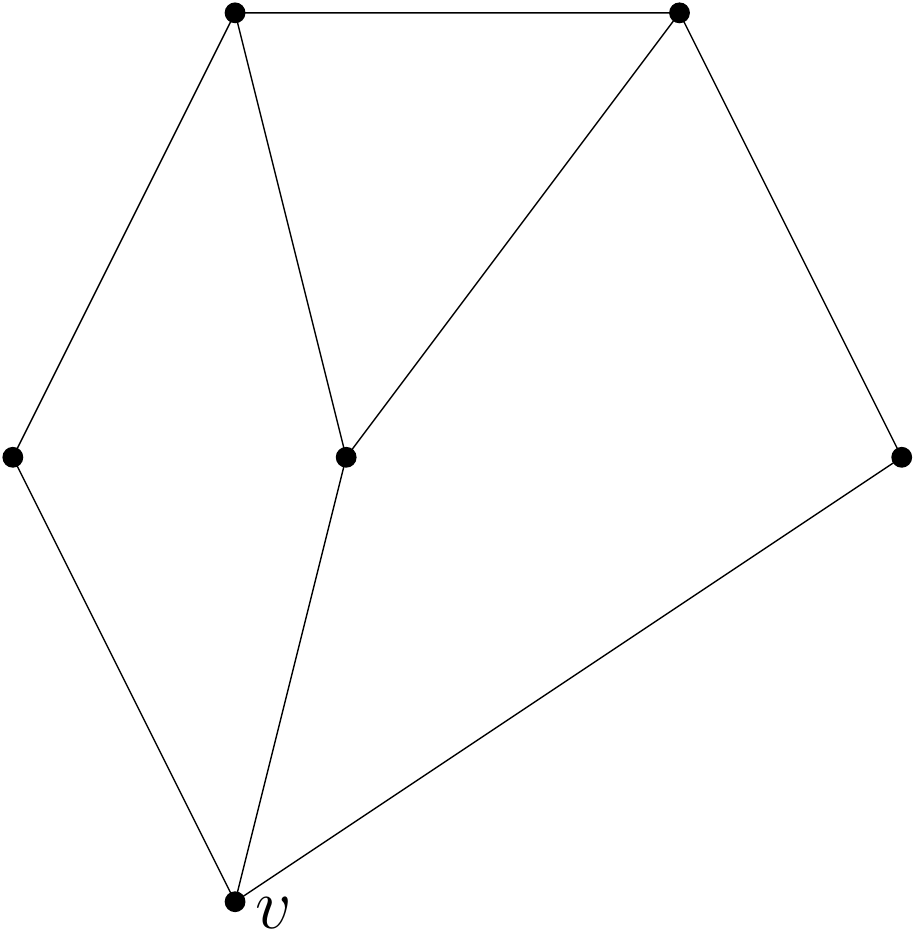} &\hspace*{0.3in} & \includegraphics[height=4cm]{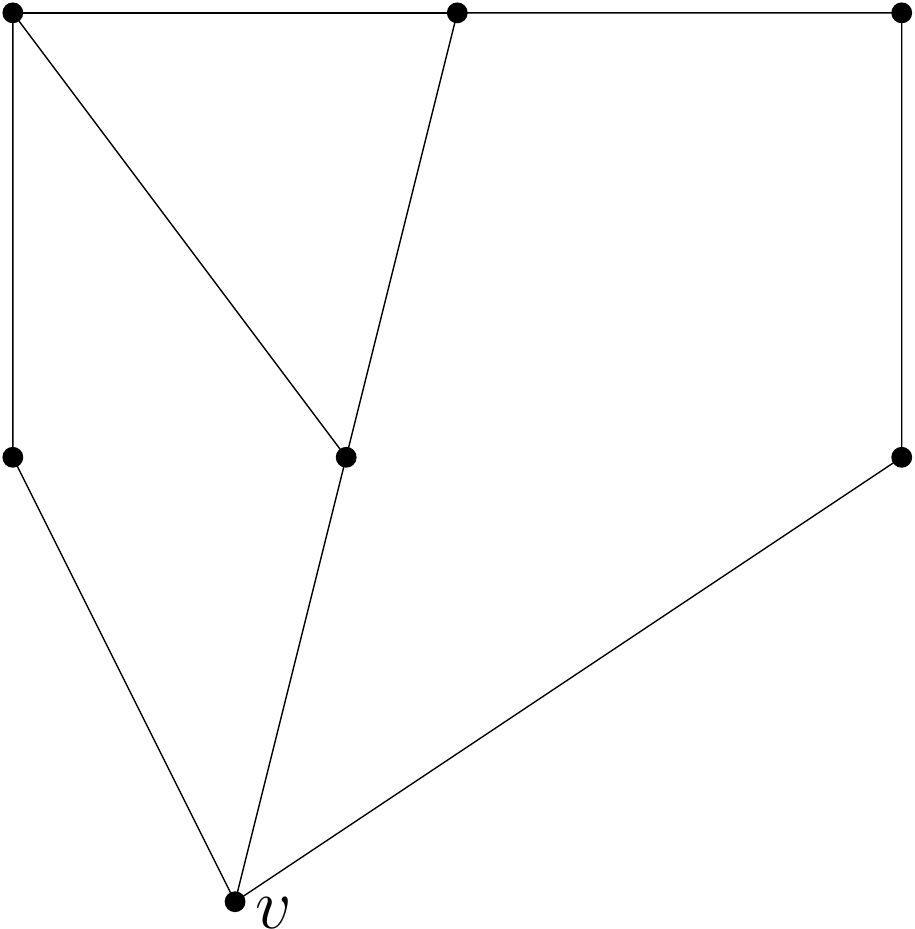} \\
(a) $G_v^a$& & (b) $G_v^b$ 
  \end{tabular}
  \caption{It is easy to check that $D_0\left(G_v^a \right) = D_0\left(G_v^a \right) = \left\{(2,0) \right\}$ and $D_1\left(G_v^a \right) = D_1\left(G_v^a \right) = \left\{(1,2),(0,2),(0,2) \right\}$. However, both the number of vertices and the number of edges in the layer-$2$ subgraphs ($G_{v;2}^a$ and $G_{v;2}^b$) are different from each other.} 
  \label{fig:cannot_recover}
\end{figure}

Unfortunately, there are a lot of counterexamples preventing us from recovering $\sigma_k$ and $\epsilon_k$ from $D_0$ and $D_1$. See Figure \ref{fig:cannot_recover} for one of them. However, if we use the persistence diagrams of the following modified filtration, we can actually recover all three quantities (see Theorem \ref{thm:decode_PD}).

\paragraph{\bf The refined filtration for rooted graphs} Given a rooted graph $G_v$, for each edge $e \in G_{v;k}$, we add its mid point $p_e$ to $G_v$ and artificially set $f_v(p_e) = k + 0.5$. We do this for all $k \in [h_v]$. Roughly speaking, we subdivide all the edges in all layer-$k$ subgraphs $G_{v;k}$. The resulting graph is denoted by $\widehat{G}$. Similarly, $\widehat{G}$ can be viewed as a finite $1$-dimensional CW complex $\widehat{X}$. We can also define a piecewise linear function $F_v': \widehat{X} \rightarrow \mathbb{R}$ such that $F_v'$ and $f_v$ agree at $\widehat{X}^{(0)}$ (the $0$-skeleton of $\widehat{X}$). The super-level set filtration induced by $F_v'$ is called the \emph{refined filtration} for graph $G_v$. See Figure \ref{fig:refined_filtration} for an illustration. Again, we consider the $0$th and $1$st persistence diagrams of the refined filtration.

For example, for graph $G_v^a$ in Figure \ref{fig:cannot_recover} (a), it is not hard to see that $D'_0 \left(G_v^a\right) = \left\{(2.5,0) \right\}$ and $D'_1 \left(G_v^a\right) = \left\{(1,2.5), (0,2), (0,2) \right\}$; for graph $G_v^b$ in Figure \ref{fig:cannot_recover} (b), it is easy to check that $D'_0 \left(G_v^b\right) = \left\{(2.5,0) \right\}$ and $D'_1 \left(G_v^b\right) = \left\{(1,2.5), (0,2), (0,2.5) \right\}$.

\begin{figure}[t]
  \centering
  \begin{tabular}{ccc}
  \includegraphics[height=5cm]{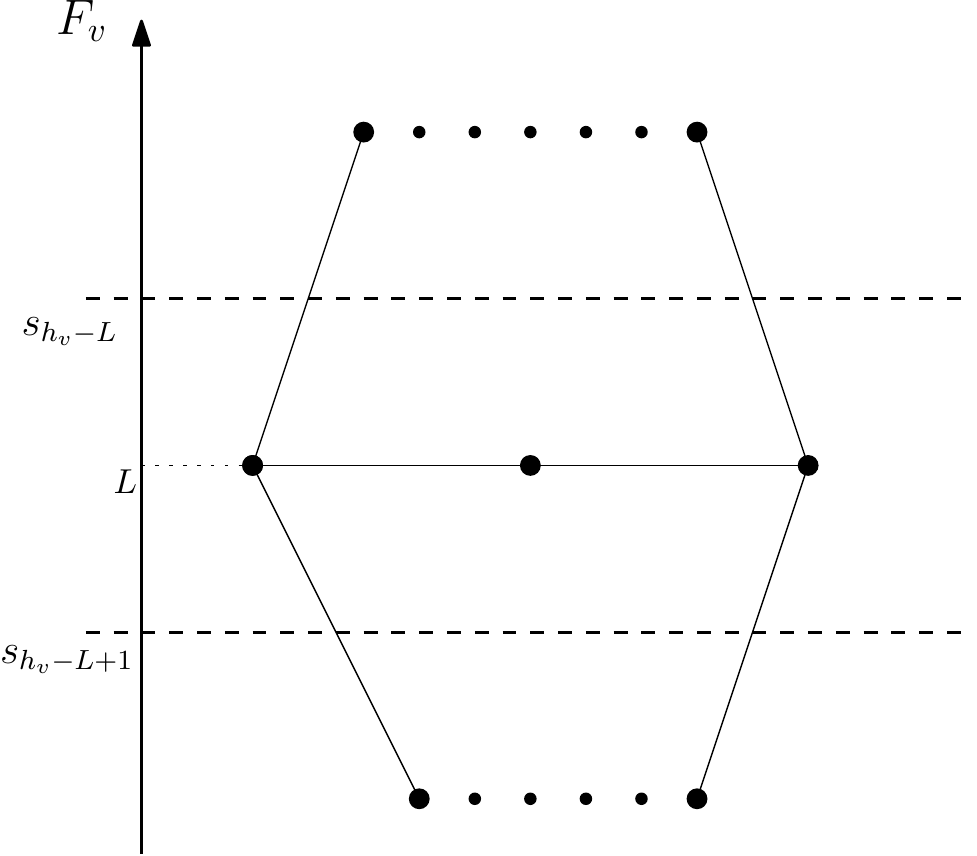} &\hspace*{0.3in} & \includegraphics[height=5cm]{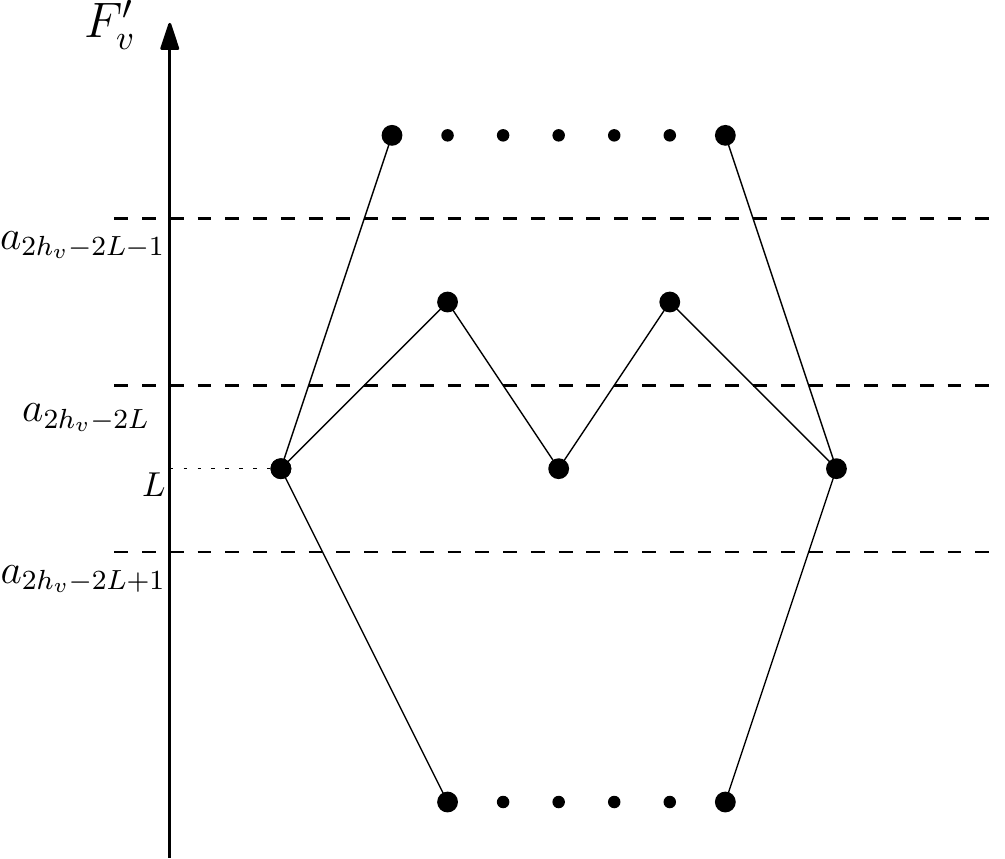} \\
(a) $X$& & (b) The corresponding $\widehat{X}$
  \end{tabular}
  \caption{(a) shows the CW complex $X$ with function $F_v$ discussed at the beginning of this section. (b) shows the refined version of (a). Here, we show a local screenshot of $G_{v;L}$ for both cases. In (a), $s_{h_v-L}$ and $s_{h_v-L+1}$ are two consecutive homological regular values chosen for constructing the filtration, while in (b), $a_{2h_v-2L-1} $, $a_{2h_v-2L}$ and $a_{2h_v-2L+1}$ are  homological regular values for constructing the refined filtration.} 
  \label{fig:refined_filtration}
\end{figure}

Similar to Claim \ref{claim:betti1}, we have the following result regarding to the $0$th and $1$st Betti numbers.
\begin{claim} 
Let $D_0',D_1'$ be the $0$th and $1$st persistence diagrams of the refined filtration induced by $F_v'$, respectively. Then, for any $k \geq 1$, we have 
\begin{align*}
\beta_0\left(G_{v;k}\right) & = \left| \left\{(x, y) \in D'_0 : x \geq k, y < k \right\} \right| + \left| \left\{(x, y) \in D'_1 : x < k, y \geq k+1 \right\} \right|\\
\beta_1\left(G_{v;k}\right) &= \left|\left\{(k,k+0.5) \in D'_1 \right\}\right|
\end{align*}
\end{claim}

Furthermore, by using the refined filtration, we can actually recover the number of vertices $\sigma_k$ and the number of edges $\epsilon_k$ for all layer-$k$ subgraphs. The following result directly follows from the Euler characteristic formula and a variant of Theorem 2 in \cite{bendich2013homology}, thus we omit the proof.

\begin{claim}\label{thm:decode_PD}
Let $D'_0, D'_1$ be the $0$th and $1$st persistence diagrams of the refined filtration induced by $F_v'$. Let $\sigma_k$ and $\epsilon_k$ be the number of vertices and the number of edges in $G_{v;k}$, respectively. Let $\omega_{k}$ be the number of $(k,k+1)$-crossings. We have the following recovery result.
\begin{enumerate}
\item For any $k \geq 1$, we have 
\begin{align*}
\sigma_k &= \left| \left\{(x, y) \in D'_0 : x \geq k, y < k \right\} \right| + \left| \left\{(x, y) \in D'_1 : x < k, y > k \right\} \right|;\\
\epsilon_k &= \left| \left\{(x, k+0.5) \in D'_1 : x \leq k\right\} \right|.
\end{align*}
\item For any $k \geq 0$, we have
\begin{align*}
\omega_k = 
\left| \left\{(x, y) \in D'_0 : x \geq k+ 1, y <k+1 \right\} \right| + \left| \left\{(x, y) \in D'_1 : x < k+1, y \geq k+1 \right\} \right|    
\end{align*}

\end{enumerate}
\end{claim}

\end{document}